\documentclass[12pt]{amsart}
\usepackage[headings]{fullpage}
\usepackage{amssymb,epic,eepic,epsfig,amsbsy,amsmath,amscd}
\usepackage[all]{xy}
                        \theoremstyle{plain}
\usepackage{mathrsfs}

         \newcommand{\psdraw}[2]
         {\begin{array}{c} \hspace{-1.3mm}
        \raisebox{-4pt}{\epsfig{figure=#1.eps,width=#2}}
        \hspace{-1.9mm}\end{array}}

\newtheorem{theorem}{Theorem}[section]
\newtheorem{thm}{Theorem}
\newtheorem{lemma}[theorem]{Lemma}
\newtheorem{corollary}[theorem]{Corollary}
\newtheorem{proposition}[theorem]{Proposition}

\newtheorem{definition}{Definition}

\theoremstyle{definition}

\newtheorem{remark}[theorem]{Remark}

\def\BC{\mathbb C}

\def\BZ{\mathbb Z}
\def\BR{\mathbb R}

\def\CS{\mathcal S}

\def\al{\alpha}
\def\ve{\varepsilon}
\def\be { \begin{equation} }
\def\ee { \end{equation} }

\newcommand\no[1]{}

\def\sdeg{\mathrm{deg}_{lr}}
\def\mm{{\mathfrak m}}

\def\ldeg{\mathrm{ldeg}}

                  \def\uu{\mathfrak u}

\begin{document}

\title{On Kauffman Bracket Skein Modules at Roots of Unity }

\author[Thang  T. Q. L\^e]{Thang  T. Q. L\^e}
\address{School of Mathematics, 686 Cherry Street,
 Georgia Tech, Atlanta, GA 30332, USA}
\email{letu@math.gatech.edu}

\thanks{Supported in part by National Science Foundation. \\
2010 {\em Mathematics Classification:} Primary 57N10. Secondary 57M25.}

\begin{abstract}
We reprove and expand results of Bonahon and Wong on central elements of the Kauffman bracket skein modules at root of 1 and
on the existence of the Chebyshev homomorphism, using elementary skein methods.
\end{abstract}

\maketitle

\setcounter{section}{-1}

\section{Introduction}
\subsection{Kauffman bracket skein modules} Let us recall the definition of the Kauffman bracket skein module, which was introduced by J. Przytycki \cite{Przy} and V. Turaev \cite{Turaev}.
 Let $R=\BC[t^{\pm1}]$.
 A {\em framed link} in an oriented $3$-manifold $M$ is a disjoint union of smoothly embedded circles, equipped with a
 non-zero normal vector field. The empty set is also considered a framed link.
 The Kauffman bracket skein module $\CS(M)$ is the $R$-module
 spanned by isotopy classes of framed links in $M$ subject to the following relations
\begin{align}
 L & = t L_+ + t^{-1} L_-   \label{eq.skein1}\\
 L \sqcup U &= -(t^2 + t^{-2}) L,  \label{eq.skein2}
\end{align}
where in the first identity, $L, L_+, L_-$ are  identical except in a ball in which they look like in Figure \ref{fig.skein},
\begin{figure}[htpb]
$$ \psdraw{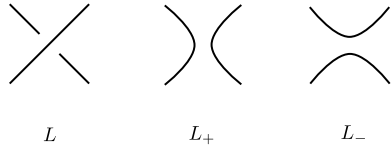}{2in} $$
\caption{The links $L$, $L_+$, and $L_-$}
\label{fig.skein}
\end{figure}
and in the second identity, the left hand side stands for  the  union of a link $L$ and the trivial framed knot $U$ in a ball disjoint from $L$. If $M=\BR^3$ then $\CS(\BR^3)=R$. The
value of a framed link $L$ in $\CS(\BR^3)=R=\BC[t^{\pm1}]$ is a version of Jones polynomial \cite{Ka}.

For a non-zero complex number $\xi$, let $\CS_\xi(M)$ be the quotient $\CS(M)/(t-\xi)$, which is
a $\BC$-vector space.

For an oriented surface $\Sigma$, possibly with boundary, we define $\CS(\Sigma):= \CS(M)$, where $M= \Sigma \times [-1,1] $ is the cylinder over $\Sigma$. The skein module
$\CS(\Sigma)$ has an algebra structure induced by the operation
of gluing one cylinder on top of the other.

For a framed knot $K$ in $M$ and a polynomial $p(z)= \sum_{j=0}^d a_j z^j \in \BC[z]$, define $p(K)$ by
$$ p(K)= \sum_{j=0}^d a_j \, K^{(j)} \in \CS(M),$$
where  $K^{(j)}$ be the link consists of $j$ parallels of $K$ (using the framing of $K$) in a small neighborhood of $K$.
When $L$ is a link, define $p(L)$ by applying $p$ to each component of $L$. More precisely, for a framed link $L\subset M$ with $m$ components $L_1,\dots,L_m$, define
$$ p(L)= \sum_{j_1,\dots,j_m =0}^d\, \left( \prod_{k=1}^m a_{j_k}\right)\left( \bigsqcup_{k=1}^m  L_k^{(j_k)}\right).$$
Here $\bigsqcup_{k=1}^m   L_k^{(j_k)}$ is the link which is the union,  over $k\in \{1,\dots, m\}$, of $j_k$ parallels of $L_k$.

\begin{remark} Suppose $K\subset \Sigma$ is a simple closed curve on the surface $\Sigma$. Consider $K$ as a framed knot in $\Sigma \times [-1,1]$ by identifying $\Sigma = \Sigma \times 0$ and equipping $K$ with the vertical framing, i.e.
the framing where the normal vector is perpendicular to $\Sigma$ and has direction from $-1$ to $1$. Then $K^{(j)}=K^j$, where $K^j$ is the power in the algebra $\CS(\Sigma)$. Thus, $p(K)$ has the usual meaning of applying a polynomial to an element of an algebra.

But if $K$ is a knot in $\Sigma \times [-1,1]$, our $p(K)$ in general is not the result of applying the polynomial $p$ to the element $K$ using the algebra structure of $\CS(\Sigma)$, i.e. $p(K) \neq \sum a_j K^j$.
\end{remark}
\subsection{Bonahon and Wong's results}

\begin{definition}
A polynomial $p(z) \in \BC[z]$ is called {\em central} at $\xi\in \BC^\times $ if for any oriented surface $\Sigma$ and any framed link $L$ in $\Sigma\times [-1,1]$, $p(L)$ is central in the algebra $\CS_\xi(\Sigma)$.

\end{definition}

Bonahon and Wong \cite{BW} showed that if $\xi$ is root of unity of order $2N$, then $T_N(z)$ is central, where $T_N(z)$ is the Chebyshev polynomial
of type 1 defined recursively by
$$ T_0(z)=2, T_1(z)=1, T_n(z) = z T_{n-1}(z) - T_{n-2}(z), \ \forall n \ge 2.$$

We will prove a  stronger version, using a different method.
\begin{thm} A non-constant polynomial $p(z)\in \BC[z]$ is central at $\xi\in \BC^\times$ if and only if

(i) $\xi$ is  a root of unity and

(ii) $p(z)\in \BC[T_N(z)]$, i.e. $p$ is a
 $\BC$-polynomial in $T_N(z)$, where $N$ is the order of $\xi^2$.
\label{thm.main1}
\end{thm}

\begin{remark}
We also  find a version of ``skew-centrality" when $\xi^{2N}=-1$ (see Section~\ref{sec.2}), which will be useful in this paper and elsewhere.
\end{remark}

\begin{remark} \label{rem.02}
Let us call a polynomial $p(z) \in \BC[z]$  {\em weakly central} at $\xi\in \BC^\times $ if for any oriented surface $\Sigma$ and any simple closed curve $K$ on $\Sigma$, $p(K)$ is central in the algebra $\CS_\xi(\Sigma)$. Then our proof will also show that Theorem \ref{thm.main1} holds true if one replaces ``central" by ``weakly central". It follows that being central is equivalent to being weakly central.
\end{remark}

\def\TT{ \mathbf{Ch}}

A remarkable result of Bonahon and Wong is the following.
\begin{thm} [Bonahon-Wong \cite{BW}]Let $M$ be an oriented 3-manifold, possibly with boundary.
Suppose $\xi^4$ is a root of unity of order $N$. Let $\ve= \xi^{N^2}$. There is a unique $\BC$-linear map $\TT: \CS_\ve(M) \to \CS_\xi(M)$
such that for any framed link $L\subset M$, $\TT(L)= T_n(L)$.
\label{thm.main2}
\end{thm}

If $M= \Sigma \times [-1,1]$, then the map $\TT$ is an algebra homomorphism. Actually Bonahon and Wong only consider the case of $\CS(\Sigma)$, but their proof works also in the case of skein modules of  3-manifolds.
In their proof, Bonahon and Wong used the theory of quantum Teichm\"uller space of Chekhov and Fock \cite{CF} and Kashaev \cite{Kashaev}, and the quantum trace homomorphism developed in earlier work \cite{BW0}.
Bonahon and Wong asked for a proof using elementary skein theory. We will present here a proof of Theorem \ref{thm.main2}, using only elementary skein theory. The main idea is to use central properties (in more
general setting) and several  operators and filtrations on the skein modules  defined by arcs.

In general, the calculation of $\CS(M)$ is difficult. For some results on knot and link complements in $S^3$, see  \cite{Le,LT,Marche}. Note that if $\xi^{4N}=1$, then $\ve= \xi^{N^2}$ is a 4-th root of 1.
In this case the $\CS_\ve(M)$ is well-known and is related to character varieties of $M$. This makes Theorem \ref{thm.main2} interesting.
At $t=-1$, $\CS_{-1}(M)$ has an algebra structure and, modulo its nilradical, is
equal to the ring of regular functions on the $SL_2(\BC)$-character variety of $M$, see \cite{Bullock,PS,BFK}. For the case when $\ve$ is a primitive 4-th root of 1, see \cite{Sikora}.

\subsection{Plan of the paper} Section \ref{sec.1} is preliminaries on Chebyshev polynomials and relative skein modules. Section \ref{sec.2} contains the proof of Theorem \ref{thm.main1}.
Section \ref{sec.3} introduces the filtrations and operators on skein modules, and Sections \ref{sec.4} and \ref{sec.5} contain some calculations which are used in
Section \ref{sec.6}, where the main technical lemma about the skein module of the twice punctured torus is proved. Theorem \ref{thm.main2} is proved in Section \ref{sec.7}.

\subsection{Acknowledgements} The author would like to thank F. Bonahon, whose talk at the conference ``Geometric Topology at Columbia University
(August 12-16, 2013) has prompted the author to work on this project. The author also thanks  C.~Frohman,  A.~Sikora, and H.~Wong for helpful discussions.
The work is supported in part by NSF.

\def\cI{\mathcal I}
\def\Id{\mathrm{Id}}
\def\BA{\mathbb A}

\section{Ground ring, Chebyshev polynomials, and relative skein modules} \label{sec.1}
\subsection{Ground ring} Let $R=\BC[t^{\pm1}]$, which is a principal ideal domain. For an $R$-module and a
 non-zero complex number $\xi\in \BC^\times$ let $V_\xi$ be the $R$-module $V/(t-\xi)$.
 Then  $R_\xi \cong \BC$ as  $\BC$-modules, and $V_\xi$ has a natural structure of an $R_\xi$-module.

We will often use the constants
\be \lambda_k:= -(t^{2k+2} + t^{-2k-2}) \in R.
\ee
For example, $\lambda_0$ is the value of the unknot $U$ as a skein element.
\no{
Suppose $V$ is a free $R$-module of finite rank. An $R$-linear operator $\varphi: V\to V$ is {\em diagonalizable} if there is a basis of $V$ in which $\varphi$ is given by
a diagonal matrix. If $\varphi$ is diagonalizable, then $\varphi_\xi: V_\xi \to V_\xi$ is a diagonalizable.

\begin{lemma} Suppose $V$ is a free $R$-module of finite rank, and
$0= V_0 \subset V_1 \subset V_2 \subset \dots \subset V_k=V$ is a nested sequence of submodules of $V$ such that $V_i/V_{i-1}$ is $R$-free for each $i=1, \dots, k$.
Assume $\varphi: V\to V$ is an $R$-operator such that each $V_i$ is an invariant submodule, and the action of $\varphi$ on $V_i/V_{i-1}$ is  a scaler operator $a_i \, \Id$, $a_i \in  R$.
Assume further that the $a_i$'s are pairwise distinct. Then $\varphi$ is diagonalizable.
\end{lemma}
}

\subsection{Chebyshev polynomials} Recall that the Chebyshev polynomials of type 1 $T_n(z)$ and type 2 $S_n(z)$ are defined by
\begin{align*}
  T_0& =2 ,  T_1(z)= z ,  T_n(z) = z T_{n-1}(z) - T_{n-2}(z) \\
 S_0& =1 ,  S_1(z)= z ,  S_n(z) = z S_{n-1}(z) - S_{n-2}(z) .
\end{align*}

Here are some well-known facts. We drop the easy proofs.

\begin{lemma} (a) 
One has
\begin{align} \label{eTn}
T_n(u + u^{-1})& = u^n + u^{-n}\\
T_n& = S_n - S_{n-2}
\end{align}

(b) For a fixed positive integer $N$, the $\BC$-span of $\{T_{Nj}, j \ge 0\}$, is $\BC[T_N(z)]$, the ring of all $\BC$-polynomials in $T_N(z)$.
\end{lemma}
Since $T_n(z)$ has leading term $z^n$, $\{T_n(z), n \ge 0\}$  is a $\BC$-basis of $\BC[z]$.



\def\cF{\mathcal F}
\subsection{Skein module of a surface} Suppose $\Sigma$ is a compact connected orientable 2-dimensional manifold with boundary.
 A knot in $\Sigma$ is {\em trivial} if it bounds a disk in $\Sigma$. Recall that $\CS(\Sigma)$ is the skein module $\CS(\Sigma \times [-1,1])$.
If $\partial \Sigma \neq \emptyset$, then $\CS(\Sigma)$ is a free $R$-module with basis the set of all links  in $\Sigma$ without trivial components, including the empty link, see \cite{PS}. Here a link in $\Sigma$ is considered as a framed link in
$\Sigma \times [-1,1]$ by identifying $\Sigma$ with $\Sigma \times 0$, and the framing at every point $P \in \Sigma \times 0$ is vertical, i.e. given by the unit positive tangent vector
of $P \times [-1,1] \subset \Sigma \times [-1, 1]$. 

The $R$-module  $\CS(\Sigma)$ has a natural $R$-algebra structure, where $L_1 \, L_2$ is obtained by placing $L_1$ on top of $L_2$.

It might happen that $\Sigma_1 \times [-1,1] \cong \Sigma_2 \times [-1, 1]$ with $\Sigma_1 \not \cong \Sigma_2$. In that case, $\CS(\Sigma_1)$ and $\CS(\Sigma_2)$ are the same as $R$-modules,
but the algebra structures may be  different.
\def\An{{\mathbb  A}}
\newcommand{\vect}[1]{\overrightarrow{#1}}
\subsection{Example: The annulus} Let $\An \subset \BR^2$ be the annulus $\An= \{  \vect x \in \BR^2,  1 \le | \vect x| \le 2\}$.
Let $z\in \CS(\An)$ be the core of the annulus, $z = \{ \vect{x}, |\vect x | = 3/2\}$.
Then $ \CS(\An) = R[z]$.


\def\cP{\mathcal P}
\subsection{Relative skein modules} {\em A marked surface} $(\Sigma, \cP)$ is a surface $\Sigma$ together with a finite set $\cP$ of points on its boundary $\partial \Sigma$. For such a marked surface,
a {\em relative framed link} is a 1-dimensional compact framed submanifold $X$ in $\Sigma \times [-1,1]$ such that $\partial X = \cP= X \cap \partial (\Sigma\times [-1,1])$, $X$ is perpendicular to $\partial (\Sigma\times [-1,1])$, and the framing at each point $P\in \cP= \partial X$ is vertical. The relative skein module $\CS(\Sigma,\cP)$ is defined as the $R$-module spanned by
isotopy class of relative framed links modulo the same skein relations \eqref{eq.skein1} and \eqref{eq.skein2}.  We will use the following fact.
\begin{proposition} \label{prop.basis}
(See \cite[Theorem 5.2]{PS}) The $R$-module $\CS(\Sigma,\cP)$ is  free  with basis the set of isotopy classes of relative links embedded in $\Sigma$ without trivial components.
\end{proposition}

\section{Annulus with two marked points and central elements} 
\label{sec.2}
\def\Aio{{\mathbb A_{io}}}
\def\An{\mathbb A}
\subsection{Marked annulus}
Recall that  $\An \subset \BR^2$ is the annulus $\An= \{  \vect x \in \BR^2\mid 1 \le | \vect x| \le 2\}$. Let  $\Aio$ be the marked surface $(\An, \{P_1, P_2\})$, with two marked points $P_1= (0,1)$, $P_2=(0,2)$, which are on different boundary components. See Figure \ref{uu}, which also depicts the  the arcs $e$, $\uu$, $\uu^{-1}$.
\begin{figure}[htpb]
$$ \psdraw{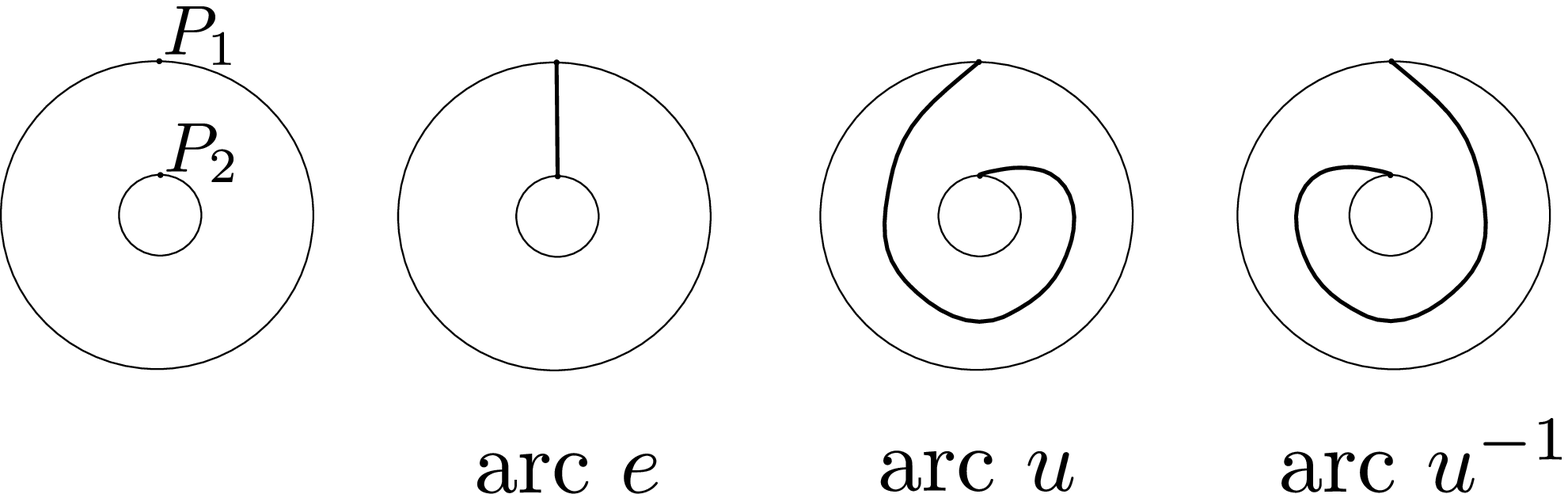}{3in} $$
\caption{The marked annulus $\Aio$, and the arcs $e$, $\uu$, $\uu^{-1}$}
\label{uu}
\end{figure}

For $L_1, L_2  \in \CS(\Aio)$ define the product $L_1 L_2 $ by placing $L_1$ inside $L_2$. Formally this means we first  shrink $\Aio \supset L_1$  by factor $1/2$, we get  $(\frac 12\Aio)  \supset (\frac 12 L_1)$, where $\frac 12\Aio $ is an annulus on the plane whose outer circle is
the inner circle of $\Aio$. Then $L_1 L_2$ is $(\frac12 L_1) \cup L_2 \subset (\frac 12\Aio) \cup \Aio$.
The identity   of $\CS(\Aio)$ is the presented by $e$, and $\uu^{-1} \uu =e = \uu \uu^{-1}$.


\begin{proposition}
The Kauffman bracket skein modules of $\Aio$ is
$ \CS(\Aio) = R[\uu^{\pm1}]$, the ring of Laurent $R$-polynomial in one variable $\uu$. In particular, $\CS(\Aio)$ is commutative.
\end{proposition}

\begin{proof}
Using Proposition \ref{prop.basis} one can easily show that the set $\{\uu^k, k \in \BZ\}$ is a free $R$-basis of $\CS(\Aio)$.
\end{proof}

\subsection{Passing through $T_k$} Recall that $\CS(\An)=R[z]$.
One defines  a left action and a right action of $\CS(\An)$ on $\CS(\Aio)$ as follows.
For $L \in \CS(\An), K \in \CS(\Aio)$ let $L \bullet K$ be the element in $\CS(\Aio)$ obtained by placing $L$ above $K$, and $K \bullet L \in \CS(\Aio)$ be the element in $\CS(\Aio)$ obtained by placing $K$ above $L$.
For example,
$$ e\bullet z= \psdraw{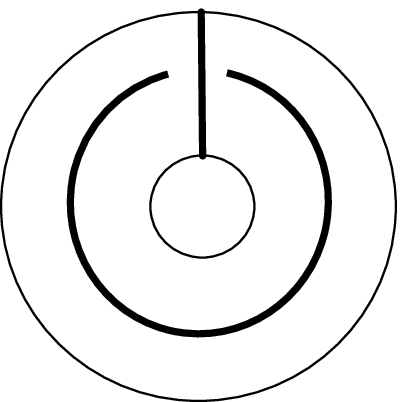}{.4in}\ , \quad z\bullet e= \psdraw{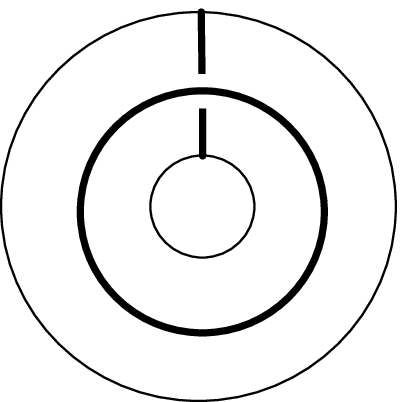}{.4in}$$

\begin{proposition} \label{prop.transp1}
One has
\begin{align}
T_k(z) \bullet e  & = t^k \uu^k + t^{-k} \uu^{-k} \label{e33a}\\
e \bullet T_k(z)   & = t^k \uu^{-k}  + t^{-k} \uu^k \label{e33b}\\
T_k(z) \bullet e - e \bullet T_k(z)  &= (t^k -t^{-k}) (\uu^k - \uu^{-k}).
\label{eq.trans1}
\end{align}

\end{proposition}
\begin{proof}
It is important to note that the map $f: \CS(\An)\to\CS(\Aio)$ given by  $f(L)= L \bullet e$ is an algebra homomorphism.

Resolve the only crossing point, we have
$$ z \bullet e =  \ \psdraw{ez}{.4in}\ =  t\  \psdraw{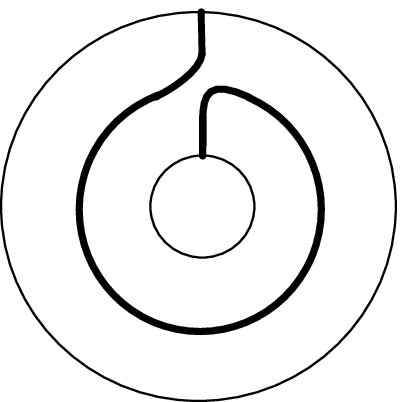}{.4in}\ + t^{-1} \psdraw{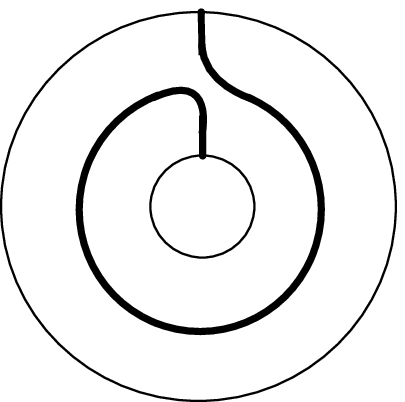}{.4in}\  = t \uu + t^{-1} \uu^{-1}.$$
Hence,
\begin{align*}
 T_k(z) \bullet e & = T_k (t \uu + t^{-1} \uu^{-1})  \quad \text{because $f$ is an algebra homomorphism}\\
 & = t^k \uu^k + t^{-k} \uu^{-k} \quad \text{ by \eqref{eTn}}.
\end{align*}
This proves \eqref{e33a}. The proof of \eqref{e33b} is similar, while \eqref{eq.trans1} follows from \eqref{e33a} and \eqref{e33b}.
\end{proof}

\begin{corollary} Suppose $\xi^{2N}=1$. Then $T_N(z)$ is central at $\xi$.
\label{cor.4}
\end{corollary}
\begin{proof} We have $\xi^N= \xi^{-N}$ since  $\xi^{2N}=1$. Then   \eqref{eq.trans1} shows that
$ T_N(z) \bullet e = e \bullet T_N(z)$, which easily implies the centrality  of $T_N(z)$.
\end{proof}
\begin{remark}
The corollary was first proved by Bonahon and Wong  \cite{BW} using another method.
\end{remark}

\def\Tp{{\mathbb T}_{\mathrm{punc}}}
\subsection{Transparent elements}
We say that $p(z)\in \BC[z]$ is {\em transparent} at  $\xi$ if for any 3 disjoint framed knots $K, K_1, K_2$ in any oriented 3-manifold $M$, $p(K)\cup K_1=p(K)\cup K_2$
in $\CS_\xi(M)$, provided that $K_1$ and $K_2$ are isotopic in $M$. Note that in general, $K_1$ and $K_2$ are not isotopic in $M \setminus K$.
\begin{proposition}
The following are equivalent

(i) $p(z) \bullet e = e \bullet p(z)$ in $\CS_\xi(\Aio)$.

(ii) $p(z)$ is transparent at $\xi$.

(iii) $p(z)$ is central at $\xi$.

\label{prop.27}
\end{proposition}
\begin{proof}
It is clear that $(i) \Rightarrow (ii) \Rightarrow (iii)$. Let us prove $(iii) \Rightarrow (i)$.
\begin{figure}[htpb]
$$ \psdraw{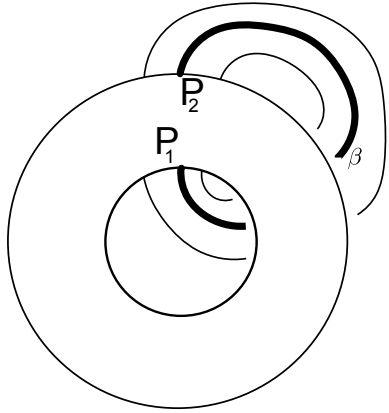}{2in} $$
\caption{ The core $\beta$ connects  $P_1$ and $P_2$ in $\Tp$}
\label{pTorus}
\end{figure}

By gluing a 1-handle to $\An$ we get a punctured torus $\Tp$  as in Figure \ref{pTorus}.
Here
the base of the 1-handle is glued to a small neighborhood of $\{P_1 \cup P_2\}$ in $\partial \An$, and the core of the 1-handle is an arc $\beta$ connecting $P_1$ and $P_2$.
Let  $\iota: \CS(\Aio) \to \CS(\Tp)$ be $R$-map which is the closure by $\beta$, i.e.  $\iota(K) = K \cup \beta$. Then $\iota(\uu^k)$ is a knot in $\Tp$ for every $k \in \BZ$,  and $\iota(\uu^k)$ is not isotopic to
$\iota(\uu^l)$ if $ k \neq l$. Since $\{ \uu^k, k \in \BZ\}$ is an $R$-basis of $\CS(\Aio)$ and the isotopy classes of links in $\Tp$ form an $R$-basis of $\CS(\Tp)$, $\iota$ is injective.

 Assume (iii). Then $ p(z) \iota (e) = \iota(e) p(z)$, or $\iota (p(z) \bullet e) = \iota ( e \bullet p(z))$. Since $\iota$ is injective, we have  $p(z) \bullet e = e \bullet p(z)$.
\end{proof}

\subsection{Proof of Theorem \ref{thm.main1}}
 The ```if" part has been proved, see Corollary \ref{cor.4}. Let us prove the ``only if part". Assume that $p(z)$ is central at $\xi$ and having degree $k\ge 1$. Since $\{T_j(z), j\ge 0\}$ is a basis of $\BC[z]$, we
can write
\be
 p(z) =  \sum_{j=0}^{k} c_j T_j(z), \quad c_j \in \BC, c_k \neq 0.
 \label{e62}
 \ee
By Proposition \ref{prop.27}, $p(z) \bullet e - e \bullet p(z)=0$.  Using expression \eqref{e62} for $p(z)$ and \eqref{eq.trans1}, we get
\begin{align*}
0 = p(z) \bullet e - e \bullet p(z) =   \sum_{j=0}^{k} c_j (\xi^{j} - \xi^{-j})(\uu^j - \uu^{-j}).
\end{align*}

Because $\{\uu^{j}, j \in \BZ\}$ is a basis of $\CS_\xi(\Aio)$, the coefficient of each $\uu^j$ on the right hand side is 0. This means,
\be
c_j =0  \quad \text{or } \  \xi^{2j} =1, \quad \forall j. \label{e33}
\ee
Since $c_k \neq 0$, we have $\xi^{2k}=1$.  Since $k \ge 1$, this shows $\xi^2$ is a root of unity of some order $N$. Then \eqref{e33} shows that $c_j=0$ unless $N|j$.
Thus, $p(z)$ is a $\BC$-linear combination of $T_{j}$ with $N |j$. This completes the proof of Theorem \ref{thm.main1}.
\subsection{Skew transparency}   One more consequence of Proposition \ref{prop.transp1} is the following.
\begin{corollary} Suppose $\xi^{2N}=-1$. Then in $\CS_\xi(\Aio)$,
$$ T_N(z) \bullet e =  - e \bullet T_N(z) .$$
\label{cor.3}
\end{corollary}
This means every time we move $T_N(K)$ passed through a component of a link $L$, the value of the skein gets multiplied by $-1$. Following is a precise statement.

Suppose $K_1$ and $K_2$ are knots in a 3-manifold $M$. Recall that an
 isotopy between $K_1$ and $K_2$ is a smooth map $H:S^1 \times [1,2] \to M$ such that for each $t \in [1,2]$, the map $H_t:S^1 \to M$ is an embedding, and
 the image of $H_i$ is $K_i$ for $i=1,2$. Here $H_t(x)= H(x,t)$. For a knot $K \subset M$ let $I_2(H,K)$ be the mod 2 intersection number of $H$ and $K$. Thus, if $H$ is transversal to $K$ then $I_2(H,K)$ is   the number of points in the
 finite set $H^{-1}(K)$ modulo 2.

\begin{definition}
Suppose $\mu= \pm 1$. A polynomial  $p(z) \in \BC[z]$ is called $\mu$-transparent at $\xi\in \BC^\times$ if for any 3 disjoint framed knots $K, K_1, K_2$ in any oriented 3-manifold $M$, with $K_1$ and $K_2$ connected by an isotopy $H$, one has
the following equality in $\CS_\xi(M)$:
$$ p(K) \cup K_1 = \mu^{I_2(H,K)} [p(K) \cup K_2].$$
\end{definition}
From Corollary \ref{cor.3} we have

\begin{corollary} Assume $\xi^{4N}=1$. Then $\mu:=\xi^{2N}=\pm 1$, and $T_N(z)$ is $\mu$-transparent.
\label{cor.5}
\end{corollary}

A special case is the following. Suppose $D \subset M$ is a disk   in $M$ with $\partial D = K$, and a framed link $L \subset M$ is disjoint from $K$. Then, if $\xi^{2N}= \mu =\pm 1$, one has
\be
 K \cup T_N(L) = \mu^{I_2(D,L)} \lambda_0 \, T_N(L) \quad \text{in \ }  \CS_\xi(M).
 \ee
Here $\lambda_0= -(\xi^2 + \xi^{-2})$ is the value of trivial knot in $\CS_\xi(M)$.

\def\BD{\mathbb D}
\def\All{\An_{oo}}
\def\cD{\mathcal D}

\section{Filtrations of skein modules} \label{sec.3}
Suppose $\Phi$ is a link in $\partial M$. We define an $R$-map $\Phi: \CS(M) \to \CS(M)$ by $\Phi(L) = \Phi \cup L$.

\subsection{Filtration by an arc} \label{sec.filtration}
Suppose $\al$ is an arc properly embedded in a marked surface $(\Sigma, \cP)$ with $\partial \Sigma \neq \emptyset$.
Assume the two boundary points of $\al$, which  are on the boundary of $\Sigma$, are  disjoint from the marked points.
Then $\cD_\al:=\al\times [-1,1]$ is a disk properly embedded in $\Sigma \times [-1,1]$, with boundary $\Phi_\al =\partial (\al \times [-1,1])=
(\al \times \{-1, 1\}) \cup (\partial \al \times [-1,1])$.

\def\fil{\mathrm{fil}}

Let $\cF^\al_k=\cF^\al_k( \CS(\Sigma) )$ be the $R$-submodule of $\CS(\Sigma)$ spanned by all relative links which intersect with $\cD_\al$ at less than or equal to $k$ points.
For $L \in \CS(\Sigma )$, we define $\fil_\al(L) = k$ if $L \in \cF^\al_k \setminus \cF^\al_{k-1}$. The filtration is compatible with the algebra structure, i.e.
$$ \fil_\al(L_1 L_2) \le \fil_\al(L_1) + \fil_\al(L_2).$$
\begin{remark}
A similar filtration was used in \cite{Marche}  to calculate the skein module of torus knot complements.
\end{remark}
A convenient way to count the number of intersection points of a link $L$ with $\cD_\al$ is to count the intersection points of the diagram of $L$ with $\al$: Let  $D$ be the vertical projection of $L$ onto $\Sigma$. In general
position  $D$ has only singular points of type double points,  and we assume further that  $D$ is transversal to $\al$.
In that case, the  number of intersection points of $L$ with $\cD_\al$ is equal to the number of intersections of $D$ with $\al$, where each intersection point of $\al$ and $D$ at a double point of $D$ is counted twice.


Recall that $\Phi_\al(L) = L \cup \Phi_\al$, where $\Phi_\al$ is the boundary of the disk $\cD_\al= \al \times [-1,1]$.
It is clear that  $\cF^\al_k$ is $\Phi_\al$ invariant, i.e. $\Phi_\al(\cF^\al_k) \subset \cF^\al_k$. It turns out that  the action of $\Phi_\al$ on the quotient $\cF^\al_k/\cF^\al_{k-1}$ is very simple.
Recall that $\lambda_k = - (t^{2k+2} + t^{-2k-2})$.
\begin{proposition}\label{eigenvalue}
For $k\ge 0$, the action of $\Phi_\al$ on $\cF^\al_k/\cF^\al_{k-1}$ is $\lambda_k$ times the identity.
\end{proposition}
 This is a consequence of Proposition \ref{prop.TL} proved in the next subsection.

\def\Sq{\mathrm{Sq}}
\subsection{The Temperley-Lieb algebra and the operator $\Phi$} The well-known Temperley-Lieb algebra $TL_k$ is the skein module of the disk with $2k$ marked points on the boundary. We will present the disk as the square
$\Sq=[0,1] \times [0,1]$ on the standard plane, with $k$ marked points on the top side and $k$-marked points on the bottom side. The product $L_1 L_2$ in $TL_k$ is defined as the result of placing $T_1$ on top of $T_2$.
The unit $\tilde e_k$ of $TL_k$ is presented by $k$ vertical straight arcs, see Figure \ref{fig.TL}.
\begin{figure}[htpb]
$$ \psdraw{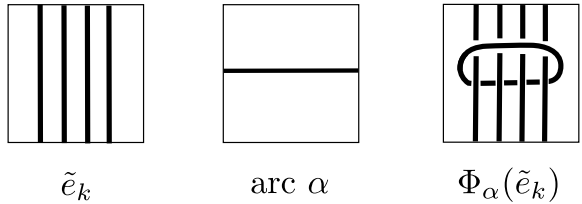}{3in} $$
\caption{The unit $\tilde e_k$, the arc $\al $, and $\Phi_\al (\tilde e_k)$. Here $k=4$.}
\label{fig.TL}
\end{figure}

Let $\al \subset \Sq$ be the the horizontal arc $[0,1] \times (1/2)$. The element $\Phi_\al(\tilde e_k)$ is depicted in Figure \ref{fig.TL}. In general, $\Phi_\al(L)$ is $L$ encircled by one simple closed curve.

\begin{proposition}\label{prop.TL}  With the above notation,
one has
\be \Phi_\al(\tilde e_k) = \lambda_k \, \tilde e_k \pmod  {\cF^\al _{k-1}}.
\label{eq.Phi}
\ee
\end{proposition}

\begin{proof} A direct proof can be carried out as follows. Using the skein relation \eqref{eq.skein1} one resolves all the crossings of the diagram of $\Phi_\al(\tilde e_k)$, and finds that only
a few terms are not in ${\cF^\al_{k-1}}$, and the sum of these terms is equal to $\lambda_k \, \tilde e_k$. This is a good exercise for the dedicated reader.

Here is another proof using more advanced knowledge of the Temperley-Lieb algebra. First we extend the ground ring to field of fraction $\BC(t)$. Then the Temperley-Lieb algebra contains
a special element called the Jones-Wenzl idempotent $f_k$ (see e.g. \cite[Chapter 13]{Lickorish}). We have $f_k = \tilde e_k \pmod  {\cF^\al _{k-1}}$, and $f_k$ is an
eigenvector of $\Phi_\al $ with eigenvalue $\lambda_k$. Hence, we have \eqref{eq.Phi}.
\end{proof}

\def\rk{\mathrm{rk}}

\no{
\begin{remark}
Note that  the scalars $\lambda_i$ are pairwise distinct. For an $R$-submodule $V \subset \CS(\Sigma)$ let $\cF^\al_m(V) = V \cap \cF^\al_m$.
It then follows easily from the proposition  that if $V$ is a finite rank $\varphi$-invariant $R$-submodule of $\CS(\Sigma)$, the action of $\varphi_\al$ on $\cF^\al_k(V)$ is diagonalizable (over the ring $R$); its eigenvalues are $\lambda_i, i=0,1,\dots,k$, with multiplicity
$\rk \cF^\al_k/\cF^\al_{k-1}$.
\end{remark}

\subsection{Filtration by a surface} Suppose $\sigma \subset X$ is a surface properly embedded in $X$. In particular, $\partial \sigma \subset \partial X$.
Let $\cF^\sigma_k\, \CS(X) (X)$ be the $R$-submodule of $\CS(X)$ spanned by all links which intersect with $\sigma$ at less than or equal to $k$ points. The filetration was used
in \cite{Marche} to calculate the skein module of torus knot complements.

Suppose $\varphi=\partial \sigma$. Then $\varphi: \CS(X) \to \CS(X)$, given by $ \varphi(\al)= \al \sqcup \varphi$ is an $R$-linear operator. It is clear that  $\varphi$ preserves the filtrations $\cF^\sigma_k(\CS(X))$.

In this paper we will focus mostly in the case when $\sigma$ is a disk.
}

\def\Aoo{\All}

\section{Another annulus with two marked points}\label{sec.4}

\subsection{Annulus with two marked points on the same boundary}
Let $\All$ be the annulus $A$ with two marked  points $Q_1,Q_2$ on the outer boundary as in Figure \ref{All}. Let $u_0, u_1$ be arcs connecting $Q_1$ and $Q_2$ in $\All$ as in Figure \ref{All}.
\begin{figure}[htpb]
$$ \psdraw{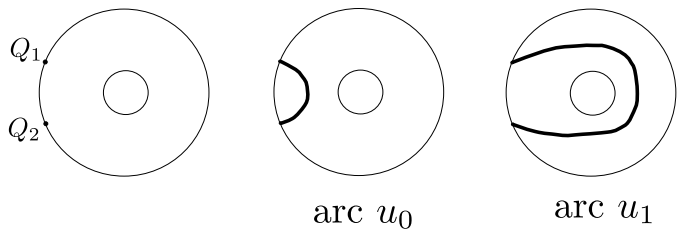}{3.5in} $$
\caption{The marked annulus $\Aoo$ and arcs $u_0, u_1$}
\label{All}
\end{figure}

Define a left $\CS(\An)$-module and a right $\CS(\An)$-module on $\CS(\Aoo)$ as follows.
For $K\in \CS(\All)$ and $L\in \CS(\An)$ let $KL$ be the skein in $\CS(\Aoo)$ obtained by placing $K$ on top of $L$, and $LK \in \CS(\Aoo)$ obtained by placing $L$ on
top of $K$. It is easy to see that $KL= LK$.
 Recall that $\CS(\BA) = R[z]$.
 \begin{proposition}
 The module $\CS(\All)$ is a free $\CS(\BA)$-module with basis  $\{ u_0,u_1\}$:
 $$ \CS(\All) = R[z]\, u_0  \oplus  R[z] \, u_1.$$
 \end{proposition}
 \begin{proof}
Any relative link in  $\Aoo$ is of the form $u_i \, z^m$ with $i=0,1$ and $m \in \BZ$. The proposition now follows from Proposition \ref{prop.basis}.
 \end{proof}

 \subsection{Framing change and the unknot} Recall that $S_k$ is the $k$-th Chebyshev polynomial of type 2.
 The values of the unknot colored by $S_k$ and the framing change are well-known (see e.g. \cite{BHMV}): In $\CS(M)$, where $M$ is an oriented $3$-manifold, one has
 \begin{align}
 L \sqcup S_k(U) = (-1)^k \frac{t^{2k+2}-t^{-2k-2} }{t^2-t^{-2}}\, L  \label{eq.U}\\
 S_k \left(  \psdraw{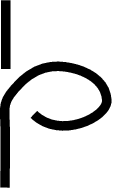}{.25in} \right ) = (-1)^k t^{k^2+ 2k} S_k  \left(\psdraw{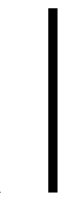}{.25in} \right).
 \label{eq.framing}
 \end{align}
 Here in \eqref{eq.U},  $U$ is the trivial knot lying in a ball disjoint from $L$.
\subsection{Some elements of $\CS(\All)$}
Let $u_k,  k\ge 0$ are arcs in $\All$ depicted in Figure \ref{fig.un}. The element $u_1$ and $u_0$  are the same as the ones defined in Figure \ref{All}. Let $v_0=u_0$ and $v_k, k\ge 1$
are arcs in $\All$ depicted in Figure \ref{fig.un}.
\begin{figure}[htpb]
$$ u_k\, = \ \psdraw{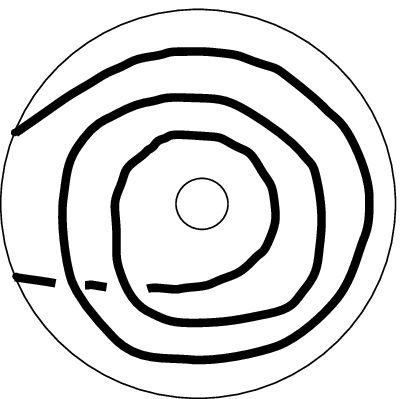}{.8in}  \qquad v_k \, = \  \psdraw{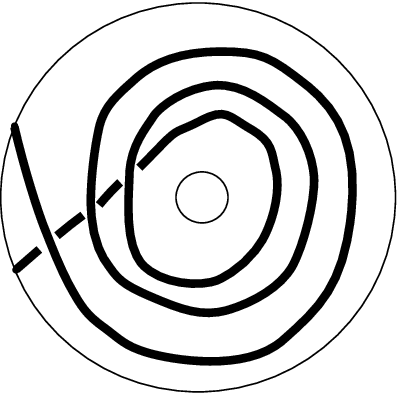}{.8in} $$
\caption{$u_k$ and $v_k$, with $k=3$}
\label{fig.un}
\end{figure}
\begin{proposition}\label{prop.un}
One has
\begin{align}
u_k & = t^{k-1} \,   S_{k-1}(z)\, u_1\,  + t^{k-3} \,  S_{k-2}(z)\, u_0\,  \quad \forall k \ge 1\label{eq:un} \\
v_k &= t^{2-k} \,  S_{k-1}(z) \,  u_1\, + t^{-k} \, S_{k}(z)\,  u_0\,  \quad \forall k \ge 0\label{eq:vn}.
\end{align}
\end{proposition}
\begin{proof}
Suppose $k \ge 3$. Apply the skein relation to the innermost crossing  of $u_k$, we get
$$ u_k\,  = \ \psdraw{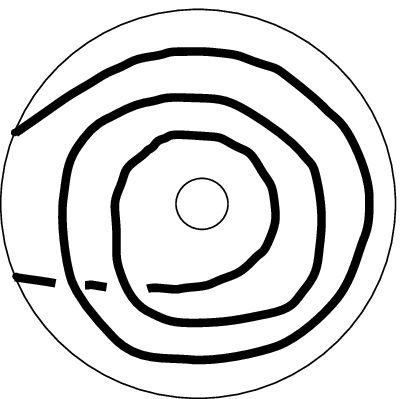}{.6 in} \ = \ t \ \psdraw{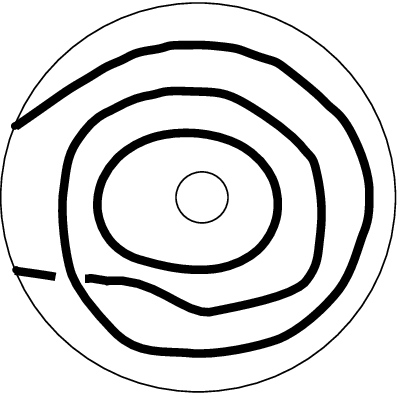}{.6 in}\  +\  t^{-1}\ \psdraw{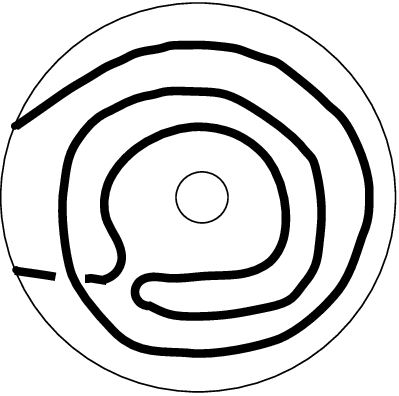}{.6 in}\,$$
which, after an isotopy and removing a framing crossing, is
$$ u_k = t \,  u_{k-1} \, z  - t^2\, u_{k-2},$$
from which one can easily prove \eqref{eq:un} by induction.

Similarly, using the skein relation to resolve the innermost crossing point of $v_k$, we get
$$ v_k =t^{-1}  v_{k-1}\, z    - t^{-2}\, v_{k-2}, \quad \text{  for $k \ge 2$,}$$
from which one can prove \eqref{eq:vn} by induction.
\end{proof}
\begin{remark}
Identity \eqref{eq:un} does not hold for $k=0$. This is due to a framing change.
\end{remark}
\subsection{Operator $\Psi$} Let $\Psi$ be the arc in $\partial \BA \times [-1,1]$ beginning at $Q_1$ and ending at $Q_2$, as depicted in Figure \ref{fig.Psi}.
\begin{figure}[htpb]
$$ \psdraw{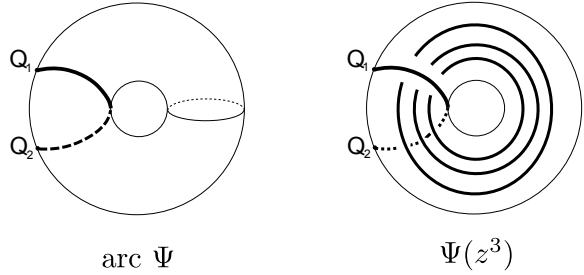}{3in} $$
\caption{Arc $\Psi$ connecting $Q_1$ and $Q_2$, and $\Psi(z^3)$}
\label{fig.Psi}
\end{figure}
Here we draw
$\BA \times [-1,1]$ as a handlebody. For any element $\al \in \CS(\BA)$ let $\Psi(\al)\in \CS(\All)$ be the skein $\Psi \cup \al$. For example, $\Psi(z^3)$ is given in
Figure \ref{fig.Psi}.
\begin{proposition}
For $k \ge 1$, one has
\be
\Psi(T_k(z))=  u_1 \left[ t^2(t^{-2k} - t^{2k})\,  \, S_{k-1}(z)\right]  + u_0 \left[  t^{-2k} S_k(z) - t^{2k} S_{k-2}(z)\right].
\label{eq.104}
\ee
\label{prop.Psi}
\end{proposition}
\begin{proof}
 Applying Proposition \ref{prop.transp1} to the part in the left rectangle  box, 
 we get
 $$ 
  \psdraw{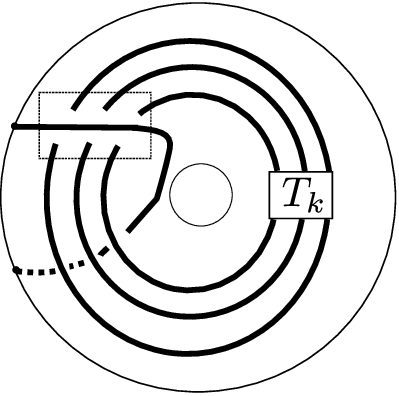}{1in} = \ t^k \, \psdraw{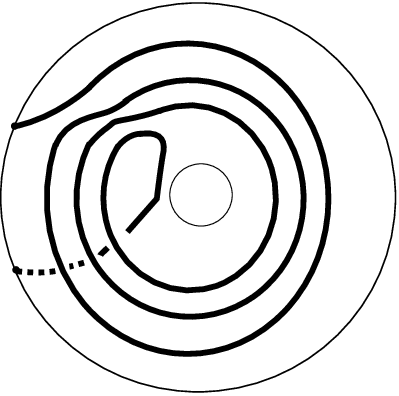}{1in} \ +\  t^{-k} \ \psdraw{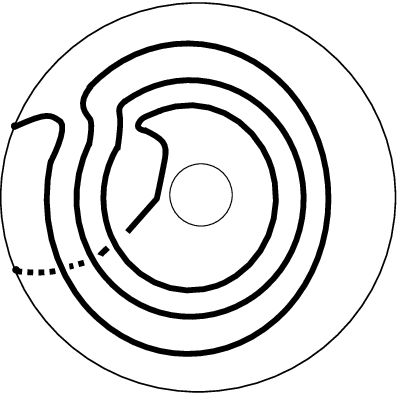}{1in} \ .$$
 The positive  framing crossing in the first term gives a factor $-t^3$. Thus,
 \begin{align*}
 \Psi(T_k(z)) & =
-t^{k+3} \, u_k + t^{-k} \, v_k.
 \end{align*}
 Plugging in the values of $u_k, v_k$ given by Proposition \ref{prop.un}, we get the result.
\end{proof}

\begin{remark}
One can use Proposition \ref{prop.Psi} to establish  product-to-sum formulas similar to the ones in \cite{FG}.
\end{remark}

\def\cH{\mathcal H}

\section{Twice punctured disk} \label{sec.5}
\subsection{Skein module of twice punctured disk}\label{sec:Phi}
 Let $\cD\subset \BR^2$ be the disk of radius 4 centered at the origin, $\cD_1 \subset \BR^2$ the disk of radius 1 centered $(-2,0)$, and
$\cD_2$ the disk of  radius 1 centered $(2,0)$.
\begin{figure}[htpb]
$$ \psdraw{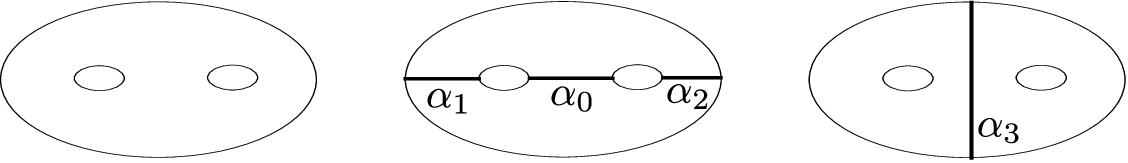}{3.5in} $$
\caption{The twice punctured disk $\BD$ and  the arcs $\al_1, \al_0, \al_2, \al_3$}
\label{arcs}
\end{figure}
We define $\BD$ to be $\cD$ with the interiors of $\cD_1$ and $\cD_2$ removed.
The horizontal axis intersects $\BD$ at 3 arcs denoted from left to right by $\al_1, \al_0, \al_2$, see Figure \ref{arcs}. The vertical axis of $\BR^2$ intersects $\BD$ at an arc denoted by $\al_3$.
 The corresponding curve $\Phi_{\al_i}$ on $\partial \BD \times [-1,1]$ will be denoted simply
by $\Phi_i$, for $i=0,1,2,3$. If $\BD \times [-1,1]$ is presented as the  handlebody $\cH$, which is a thickening of $\BD$ in  $\BR^3$, then the curves $\Phi_1,\Phi_0,\Phi_2, \Phi_3$ look like in Figure \ref{curves_alpha}.
\begin{figure}[htpb]
$$ \psdraw{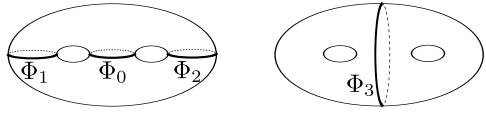}{2.5in} $$
\caption{The curves $\Phi_1,\Phi_0,\Phi_2,\Phi_3$ on the boundary of the handlebody}
\label{curves_alpha}
\end{figure}

Let $x_1,x_2$, and $y$ be the closed curves in $\BD$:
$$  x_1= \psdraw{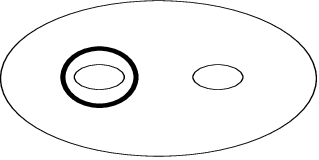}{1in} \ , \quad x_2 =  \psdraw{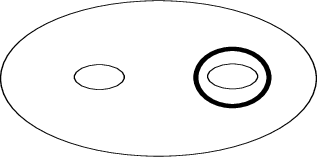}{1in} \ , \quad y = \psdraw{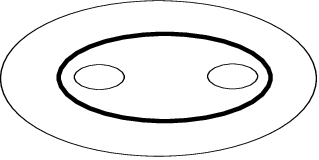}{1in}$$

It is known that $\CS(\BD) = R[x_1,x_2,y]$, the $R$-polynomial in the variables  $x_1, x_2, y$, see \cite{BP}. In particular, $\CS(\BD)$ is commutative.

Let $\sigma$ be the  rotation about the origin of $\BR^2$ by $180^0$. Then $\sigma(\BD)= \BD$. Hence $\sigma$ induces an automorphism of $\CS(\BD)=R[x_1, x_2, y]$, which
is an algebra automorphism. One has $\sigma(y)=y, \sigma(x_1)= x_2, \sigma(x_2)= x_1$.

\def\rdeg{\mathrm{deg}_r}
\def\ldeg{\mathrm{deg}_l}
\subsection{Degrees on $\CS(\BD)=R[x_1,x_2,y]$}
Define the left degree, right degree, and double degree on $R[x_1,y, x_2]$ as follows.
For a monomial $\mm=x_1^{a_1} y^b x_2^{a_2}$ define its left degree $\ldeg(\mm)= a_1 + b$, right degree $\rdeg(\mm)= a_2+b$, double degree $\sdeg(\mm)=\ldeg(\mm) + \rdeg(\mm)= a_1 + a_2 + 2b$.
One readily finds that
$$ \ldeg(\mm) = \fil_{\al_1}(\mm), \quad \rdeg(\mm) = \fil_{\al_2}(\mm),$$
where $\fil_\al$ is defined in Section \ref{sec.filtration}. Using the definition of $\fil_\al$ involving the numbers of intersection points we get the following.

\begin{lemma} \label{degree_bound}
Suppose $L$ is an embedded  link in $\BD$ and $L$ intersects transversally the arc $\al_i$ at $k_i$ points for $i=1,2,3$.
Then, as an element  of $\CS(\BD)$, $L= x_1^{a_1}x_2^{a_2} y^b$, where  $2b \le k_3$ and
\begin{align*}
\ldeg(L) \le k_1 , \quad \ldeg(L) & \equiv k_1 \pmod 2 \\
\rdeg(L)\le k_2 , \quad  \rdeg(L) & \equiv k_2 \pmod 2.
\end{align*}
Consequently, $\sdeg(L)\le k_1 + k_2$ and $\sdeg(L) \equiv k_1 + k_2 \pmod 2$.
\end{lemma}

\begin{proof} If $L= L_1 \sqcup L_2$ is the union of 2 disjoint sub-links, and the statement holds for each of $L_i$, then it holds for $L$.
Hence we assume $L$ has one component, i.e.
$L$ is an embedded loop in $\BD \subset \BR^2$. 
Then $L$ is isotopic to either a trivial loop, or $x_1$, or $x_2$, or $y$. In each case, the statement can be verified easily.
For example, suppose
 $L= x_1$. For the  mod 2 intersection numbers, $I_2(L,\al_1) =I_2(x_1,\al_1)=1$.
Hence $k_1$, the geometric intersection number between $L$ and $\al_1$, must be odd and bigger than or equal to 1. Hence, we have
$\ldeg(L) \le k_1$ and  $\ldeg(L)\equiv k_1 \pmod 2$.
\end{proof}
\begin{corollary} \label{Cor.d1}
Suppose $L$ is a link diagram on $\BD$ which intersects transversally the arc $\al_i$ at $k_i$ points for $i=1,2,3$.
Then, as an element in $\CS(\BD)$,
$$\ldeg(L) \le k_1, \quad \rdeg(L) \le k_2, \quad 2 \deg_y(L)\le k_3,$$
 and $L$ is a linear $R$-combination of
monomials whose double degrees are equal to $k_1+ k_2$ modulo 2.
\end{corollary}

\def\Err{E}

\def\xx{E}

\subsection{The $R$-module $V_n$ and the  skein $\gamma$} Let $\gamma$ and $\bar \gamma$ be the following link diagrams on $\BD$,
\be\gamma=  \psdraw{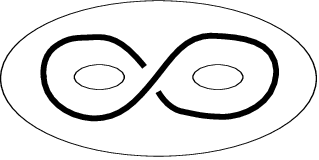}{1 in}\ , \qquad \bar \gamma=  \psdraw{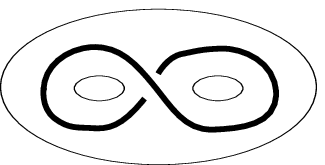}{1 in}\ .
\label{eq.gamma}
\ee
Let  $$ V_n= \{ p \in R[x_1, x_2, y] \mid \ldeg(p) \le n, \rdeg(p) \le n, \sdeg(p) \text{ even} \}.$$
In other words,  $V_n \subset R[x_1,x_2,y]$ is the $R$-submodule spanned by $x_1^{a_1} x_2^{a_2} y^b$, with $a_i+b \le n$ for $i=1,2$, and $a_1+ a_2$ even.
\begin{lemma} One has $T_n(\gamma), T_n(\bar \gamma) \in V_n$.
\end{lemma}

\begin{proof} The diagram  $\gamma^k$
 has $k$ intersection points with each of $\al_1$ and $\al_2$. By Corollary \ref{Cor.d1}, $\ldeg(\gamma^k) \le k, \rdeg(\gamma^k) \le k$, and
each monomial of $\gamma^k$ has double degree $\equiv k+ k\equiv 0 \pmod 2$. This means  $\gamma^k \in V_k$ for every $k \ge 0$.
Because $T_n(\gamma)$ is $\BZ$-linear combination of $\gamma^k$ with $ k \le n$, we have $T_n(\gamma) \in V_n$. The proof for $\bar \gamma$ is similar.
\end{proof}

\begin{remark}
It is an easy exercise to show that $T_N(\bar \gamma)= T_N(\gamma)\big|_{t\to t^{-1}}$.
\end{remark}

\section{Skein module of  twice puncture disk at root of 1} \label{sec.6}
Recall that $\gamma$ and $\bar \gamma$ are knot diagrams on $\BD$ defined by \eqref{eq.gamma}.
The following was proved by Bonahon and Wong, using  quantum Teichm\"uller algebras and their representations.

\begin{proposition}\label{prop.main1}
Suppose $\xi^4$ is a  root of 1 of order $N$. Then in $\CS_\xi(\BD)$ one has
\begin{align}
T_N(\gamma)&= \xi^{-N^2} T_N(y) + \xi ^{N^2} T_N(x_1)\, T_N( x_2)  \label{eq:main} \\
T_N(\bar\gamma)&= \xi^{N^2} T_N(y) + \xi ^{-N^2} T_N(x_1)\, T_N( x_2) \label{eq:main2}.
\end{align}
\end{proposition}
As mentioned above, there was an urge to find a proof using elementary skein theory; one such is   presented here.
Our proof roughly goes as follows. Using the transparent property of $T_N(\gamma)$, we show that $T_N(\gamma)$
is a common eigenvector of several operators. We then prove that the space of common eigenvectors has dimension at most 3, with a simple basis. We then fix coefficients of $T_N(\gamma)$ in this
basis using calculations in highest order. Then the result turns out to be the right hand side of \eqref{eq:main}.

Throughout this section we fix a complex number $\xi$ such that $\xi^4$ is a root of unity of order $N$.
 Define $\ve= \xi^{N^2}$. We will write $V_{N,\xi}$ simply by $V_N$ and $\lambda_k$ for $\lambda_k(\xi)$. Thus, in the whole section,
$$ \lambda_k= -(\xi^{2k+2} + \xi^{-2k-2}).$$

\def\cH{\mathcal H}

\subsection{Properties of $\xi$ and  $\lambda_k$} Recall that $\xi^4$ is a root of 1 of order $N$.
\begin{lemma} \label{lem.lambda}
Suppose  $ 1\le k \le N-1$. Then

(a) $\lambda_{2k} =  \lambda_0$ if and only if $k=N-1$.

(b)  $\lambda_k = \xi^{2N} \lambda_0 $ implies that $k=N-2$.

 (c) If  $N$ is even then $\xi^{2N}=-1$.

 (d) One has
 \be \xi^{2N^2 + 2N} = (-1)^{N+1}.
 \label{eq.xi}
 \ee

\end{lemma}
\begin{proof} (a) With $\lambda_k = -(\xi^{2k+2} + \xi^{-2k-2})$, we have
$$\lambda_{2k} -  \lambda_0=  -\xi^{-2-4k} (\xi^{4k} -1) (\xi^{4k+4} -1).$$
Hence, $\lambda_{2k} -  \lambda_0= 0$ if and only if either $N|k$ or $N| (k+1)$. With $1\le k\le N-1$, this is equivalent to $k=N-1$.

 (b)  We have
$$\lambda_{k} -  \xi^{2N} \lambda_0=  -\xi^{-2N-2}(\xi^{2N-2k} -1) (\xi^{2N+2k+4} -1).$$
Either (i) $\xi^{2N-2k} =1$ or (ii) $\xi^{2N+2k+4} =1$. Taking the squares of both identities, we see that either $N | (N-k)$ or $N | (k+2)$. With $ 1\le k \le N-1$, we
conclude that $k=N-2$.

(c) Suppose $N$ is even. Since $\xi^4$ has order $N$, one has $(\xi^4)^{N/2}=-1$. Then $\xi^{2N} = (\xi^4)^{N/2} =-1$.

(d) is left for the reader.
\end{proof}

\subsection{Operators $\Phi_i$ and  the vector space $W$} \label{sec.W}
Recall that  $\Phi_i:= \Phi_{\al_i}$, $i=0,1,2,3$, is defined in section \ref{sec:Phi}.
Then  $\Phi_i(V_N) \subset V_N$ for $i=0,1,2,3$.

Let $\Phi_4$ be the curve on $\partial \BD \times [-1, 1]$ depicted in Figure \ref{Phi4}.
\begin{figure}[htpb]
$$ \psdraw{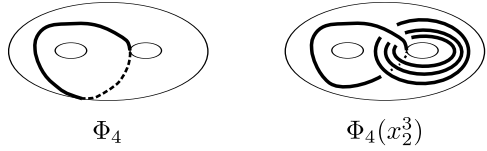}{2.5in} $$
\caption{The curve $\Phi_4$ and $\Phi_4(x_2^3)$}
\label{Phi4}
\end{figure}
Here we draw $\cH=\BD \times [-1, 1]$ as a handlebody.
We also depict $\Phi_4(x_2^3)$.

We don't have $\Phi_4(V_N)  \subset V_N$, since $\Phi_4$ in general increases the
double degree.
By counting the intersection points with $\al_1$ and $\al_2$, we have, for every $E \in \CS_\xi(\BD) = \BC[x_1, x_2, y]$,
\be
\sdeg(\Phi_4(\xx) ) \le \sdeg(\xx)+1.
\ee

\begin{proposition} If $\xx$ is one of $\{T_N(\gamma), T_N(\bar \gamma), T_N(y),T_N(x_1) T_N(x_2)\}$,  then one has
\begin{align}
\sigma(\xx)& = \xx \label{eq:eigen} \\
\Phi_1 (\xx) & = \xi^{2N}\, \lambda_0\, \xx    \label{eq:eigen1} \\
\Phi_i (\xx) & =  \lambda_0\, \xx   \quad \text{for } \ i=0,3 \label{eq:eigen3},  \\
\Phi_4(\xx) &=
\xi^{2N} x_1 \, \xx.
\label{eq:eigen4}
\end{align}
\end{proposition}
\begin{proof}
The first identity follows from the fact that each of $\gamma, \bar \gamma, y, x_1 \cup x_2$ is invariant under $\sigma$. The remaining identities follows from the $\xi^{2N}$-transparent property of $T_N(z)$, Corollary \ref{cor.5}.
\end{proof}
\begin{remark}
Note that $\Phi_4$ is $\BC[x_1]$-linear and \eqref{eq:eigen4} says that $E$ is a $\xi^{2N} x_1 $-eigenvector of $\Phi_4$.
\end{remark}

Let $W$ be the subspace of $V_N$ consisting of  elements satisfying \eqref{eq:eigen}--\eqref{eq:eigen4}. This means, $W \subset V_N$ consists of elements which are at the same time
1-eigenvector of $\sigma$, $\xi^{2N} \lambda_0$-eigenvector of $\Phi_1$, $\lambda_0$-eigenvector of $\Phi_0$ and $\Phi_3$, and $\xi^{2N} x_1 $-eigenvector of $\Phi_4$.

 We will show that $W$ is spanned by  $T_N(y), T_N(x_1) T_N(x_2)$, and possibly $1$.

\def\coeff{\mathrm{coeff}}

\subsection{Action of $\Phi_3$, $\Phi_0$, and $\Phi_1$} 
For an element $F \in \BC[x_1, x_2, y]$ and a monomial $\mm= x_1^{a_1} x_2^{a_2} y^b$ let
$\coeff(F, \mm)$ be the coefficient of $\mm$ in $F$.

\begin{lemma} 
Suppose $\xx \in W$ and $\coeff(\xx,y^N)=0$. Then $\xx \in \BC[x_1, x_2]$.
\label{lem.y-deg}
\end{lemma}

\begin{proof} 
 Let $k$ be the $y$-degree of $E$. Since $E\in W$ and $\coeff(E,y^N)=0$, one has $k \le N-1$.

 We need to show that $k=0$.
Suppose the contrary that $1\le k$. Then  $1\le k \le N-1$.

First we will prove $k=N-1$, using the fact that $\xx$ is a $\lambda_0$-eigenvector of $\Phi_3$ by \eqref{eq:eigen3}.

Recall that $\fil_{\al_3}$ is twice the $y$-degree. One has  $\fil_{\al_3}(E)=2k$. Thus  $\xx \neq 0 \in \cF^{\al_3}_{2k} /\cF^{\al_3}_{2k-1}$.
By Proposition \ref{eigenvalue}, any non-zero element in $\cF^{\al_3}_{2k} /\cF^{\al_3}_{2k-1}$ is an eigenvector of $\Phi_3$ with eigenvalue $\lambda_{2k}$. But $\xx$ is an eigenvector of $\Phi_3$ with eigenvalue $\lambda_0$. It follows that $\lambda_{2k} = \lambda_0$. By Lemma \ref{lem.lambda}, we have $k=N-1$.

 Because $\sdeg(\xx)$ is even and $\le 2N$, we must have
$$ \xx= y^{N-1} (c_1 x_1 x_2 + c_2)  + O(y^{N-2}), \quad c_1, c_2 \in \BC.$$
We will prove $c_1=0$ by showing that otherwise, $\Phi_0$ will increase the $y$-degree.
Note that  $\Phi_0$ can increase the $y$-degree by at most 1, and $\Phi_0$ is $\BC[y]$-linear. We have
\be
 \Phi_0(\xx)= y^{N-1} (c_1 \Phi_0( x_1 x_2) + c_2 \Phi_0(1)) +  O(y^{N-1}).
 \label{e52}
 \ee
The diagram of $\Phi_0(x_1 x_2)$ has 4 crossings, see Figure \ref{Phi0x1x2}.
\begin{figure}[htpb]
$$ \psdraw{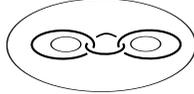}{1in} $$
\caption{$\Phi_0(x_1 x_2)$}
\label{Phi0x1x2}
\end{figure}
A simple calculation shows
$$ \Phi_0( x_1 x_2) = (1-t^4)(1-t^{-4}) y + O(y^0) .$$
Plugging this value in \eqref{e52}, with $\Phi_0(1)= \lambda_0\in \BC$,
\be
 \Phi_0(\xx)= y^{N} c_1  (1-t^4)(1-t^{-4})+  O(y^{N-1}).
 \ee
If $c_1\neq 0$, then the $y$-degree of $\Phi_0(\xx)$ is $N$, strictly bigger than that of $\xx$, and $\xx$ cannot be an eigenvector of $\Phi_0$. Thus $c_1=0$.

One has now
\be \xx = c_2 y^{N-1} + O(y^{N-2})
\label{e47}
\ee
Since the $y$-degree of $\xx$ is $N-1$, one must have  $c_2 \neq 0$.
By counting the intersections with $\al_3$, we see that $\Phi_1$ does not increase the $y$-degree. We have
\begin{align*}
 \Phi_1(\xx)& = c_2 \Phi_1(y^{N-1}) + O(y^{N-2}) \\
 &= c_2 \, \lambda_{N-1} y^{N-1} + O(y^{N-2}) \quad \text{by Proposition \ref{eigenvalue}}.
 \end{align*}
Comparing the above identity with \eqref{e47} and using the fact that $\xx$ is a $\xi^{2N} \lambda_0$-eigenvector of $\Phi_1$, we have
$$ \lambda_{N-1} = \xi^{2N} \lambda_0,$$
which is impossible since Lemma \ref{lem.lambda} says that $\lambda_k = \xi^{2N} \lambda_0$ only when $k=N-2$.
This completes the proof of the lemma.
\end{proof}

\def\vfive{\varphi_r}
\def\vsix{\varphi_l}

\def\cH{\mathcal H}

\subsection{Action of $\Phi_4$}

Recall that $\Phi_4$ is the curve on the boundary of the handlebody $\cH$ (see  Figure \ref{Phi4}) which acts on $\CS_\xi(\BD)= \BC[x_1, x_2, y]$.
The action of $\Phi_4$ is $\BC[x_1]$-linear, and every element of $W$ is a $\xi^{2N} x_1$-eigenvector of $\Phi_4$.

Recall that $\rdeg= \fil_{\al_2}$, and $\rdeg(x_1^{a_1} x_2^{a_2} y^b)= a_2 + b$. Note that  for $F \in \BC[x_1, x_2]$, $\rdeg(F)$ is exactly the $x_2$-degree of $F$.
By looking at  the intersection with $\al_2$, we see that $\Phi_4$ preserves the $\al_2$-filtration, i.e.
$ \rdeg \Phi_4( F) \le \rdeg(F)$. We will study actions of $\Phi_4$ on the associated graded spaces.

We will use the notation $ F+  \rdeg\text{-$l.o.t.$}$  to mean $F+F_1$, where $\rdeg(F_1) < \rdeg(F)$.

\begin{lemma} Suppose $1 \le k \le N-1$. One has
\begin{align}
  \Phi_4(a(x_1)\, T_N(x_2)) &= \xi^{2N} x_1 \,[a(x_1) T_N(x_2)] \label{eq.93}\\
\Phi_4(T_k(x_2)) &= y\left [ \xi^2(\xi^{-2k} - \xi^{2k}) x_2^{k-1} \right]   + \rdeg\text{-l.o.t}. \mod{\BC[x_1, x_2]}.
\label{eq.96}   
\end{align}

\end{lemma}
\begin{proof}
Identity \eqref{eq.93} follows from the $\xi^{2N}$-transparency of $T_N(z)$.

Let us prove \eqref{eq.96}. Apply identity \eqref{eq.104} to the dashed box below, we have
\begin{align*}
\Phi_4(T_k(x_2)) & = \ \psdraw{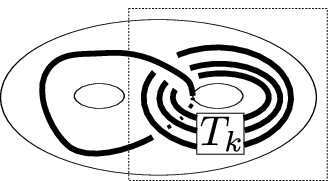}{1.5in}   \\
&= y\left [ \xi^2(\xi^{-2k} - \xi^{2k}) S_{k-1}(x_2) \right]   +  x_1 \left[ \xi^{-2k}  S_k(x_2) - \xi^{2k} S_{k-2}(x_2)  \right],
\end{align*}
which implies \eqref{eq.96}.
\end{proof}

\subsection{The space $W \cap \BC[x_1, x_2]$}
\begin{lemma} Suppose $\xx \in W \cap \BC[x_1, x_2]$ and the coefficient of $x_1^N x_2^N$ in $\xx$ is 0.
Then $E \in \BC$.
\label{lem.x1x2}
\end{lemma}
\begin{proof} Since $T_k(x_2)$ is a basis of $\BC[x_2]$, we can write $E$ uniquely as
$$ E = \sum_{k=0}^N a_k(x_1) T_k(x_2), \quad a_k(x_1) \in \BC[x_1].$$
Let $j$ be the $x_2$-degree of $E' := E - a_N(x_1) T_N(x_2)$. Then $j \le N-1$

First we will show that  $j=0$. Assume the contrary $ j \ge 1$. Thus $ 1\le j \le N-1$.
Note that $E$, by assumption, and $a_N(x_1) T_N(x_2)$, by \eqref{eq.93},  are eigenvectors of $\Phi_4$ with eigenvalue $\xi^{2N} x_1$.
It follows that $E'$ is also an eigenvector of $\Phi_4$ with eigenvalue $\xi^{2N} x_1$.
We have
\begin{align*}
 E' = \sum_{k=0}^j a_k(x_1)  \Phi_4(T_k(x_2))= a_j(x_1) T_j(x_2)  + \rdeg\text{-l.o.t}.
\end{align*}
Using  \eqref{eq.96} and the fact that $\Phi_4$ does not increase $\rdeg$, we have
\begin{align*}
\Phi_4(E') &= y \left [ a_j(x_1) \xi^2(\xi^{-2k} - \xi^{2k}) x_2^{k-1}   \right]  + \rdeg\text{-l.o.t}. \mod{\BC[x_1, x_2]}.
\end{align*}
 When $1 \le j \le n-1$, the coefficient of $y$, which is the element in the square bracket, is non-zero. Thus
$\Phi_4(E') \not \in \BC[x_1, x_2]$, while $E'\in \BC[x_1, x_2]$. This means $E'$ can not be an eigenvector of $\Phi_4$, a contradiction.
This proves $j=0$.

So we have
$$ E= a_N(x_1) T_N(x_2) + a_0(x_1).$$

Because the $\sdeg(E)  < 2N$, the $x_1$-degree of $a_N(x_1)$ is $< n$. Using the invariance under $\sigma$, one sees that $E$ must be of the form

\be E= c_1(T_N(x_1) + T_N(x_2)) + c_2, \quad c_1, c_2 \in \BC.
\label{e66}
\ee
To finish the proof of the lemma, we need to show that $c_1=0$.
Assume that $c_1 \neq 0$. Since $E$ has even double degree, $N$ is even. By Lemma \ref{lem.lambda}(c), $\xi^{2N}=-1$.

Recall that $E$ is $\lambda_0$-eigenvector of $\Phi_0$. Apply $\Phi_0$ to \eqref{e66},
$$ \lambda _0\left[ c_1(T_N(x_1) + T_N(x_2)) + c_2\right] = \Phi_0(c_1(T_N(x_1) + T_N(x_2)) + c_2).$$
Both $T_N(x_1)$ and $T_N(x_2)$ are eigenvectors of $\Phi_0$ with eigenvalues $\xi^{2N}\lambda_0= - \lambda_0$, while $\Phi_0(1)= \lambda_0$. Hence we have
$$ \lambda _0\left[ c_1(T_N(x_1) + T_N(x_2)) + c_2\right] = \lambda_0 \left[ - c_1(T_N(x_1) - T_N(x_2)) + c_2 \right].$$
which is impossible since $c_1 \lambda _0\neq 0$.

Hence, we have $c_1=0$, and $E \in \BC$.
\end{proof}
\def\coeff{\mathrm{coeff}}
\subsection{Some maximal degree parts of $T_N(\gamma)$}
\begin{lemma} One has
\begin{align}
\coeff(T_N(\gamma), y^N)&= \xi^{-N^2} \\
\coeff(T_N(\gamma), x_1^N x_2^N)& = \xi^{N^2}.
\end{align}
\label{lem.degree}
\end{lemma}
\begin{proof} Since $T_N(\gamma) = \gamma^N +  \sdeg\text{-}l.o.t.$, we have
$$ \coeff(T_N(\gamma), y^N) = \coeff(\gamma^N, y^N), \quad \coeff(T_N(\gamma), x_1^N x_2^N) = \coeff(\gamma^N, x_1^N x_2^N).$$

There are $N^2$ crossing points in the diagram of $\gamma^N$. Each crossing can be smoothed in two ways. The positive smoothing acquires a factor $t$ in the skein relation, and
the negative smoothing acquires a factor $t^{-1}$. The are $2^{N^2}$ smoothings of $\gamma^N$. Each smoothing $s$ of all the $N^2$ crossings gives rise to a link $L_s$ embedded in $\BD$.
Then $\gamma^N$ is a linear combination of all $L_s$. We will show that the only $s$ for which $L_s= y^N$ is the all negative smoothing.

Consider a crossing point $C$ of $\gamma^N$. The vertical line passing through $C$ intersect $\BD$ in an interval $\al_3'$ which is isotopic to $\al_3$, and $\fil_{\al_3}= \fil_{\al_3'}$.
 For an embedded link $L$ in $\BD$, as an element of $\CS(\BD) =
R[x_1,x_2, y]$, $L$ is a monomial whose $y$-degree is bounded above by half the number of intersection points of $L$ with $\al_3'$. The diagram $\gamma^N$ has exactly $2N$ intersection points
with $\al_3'$, with $C$ contributing two (of the $2N$ intersection points). If we positively smooth $\gamma^N$ at $C$, the result is a link diagram with $2N-2$ intersection points with $\al_3'$, and no
matter how we smooth other crossings, the resulting link will have less than or equal to $2N-2$ intersection points with $\al_3'$. Thus we cannot get $y^N$ if any of the crossing is smoothed positively.
The only smoothing which results in $y^N$ is the all negative smoothing. The coefficient of this smoothing is $\xi^{-N^2}$.

Similarly, one can  prove that the only smoothing which results in $x_1^N x_2 ^N$ is the all positive smoothing, whose coefficient is $\xi^{N^2}$.
\end{proof}

\subsection{Proof of Proposition \ref{prop.main1}}
Let
$$E= T_N(\gamma) - \xi^{N^2} T_N(x_1) T_N(x_2) - \xi^{-N^2} T_N(y).$$
Then $E \in W$. Lemma \ref{lem.degree} shows that $\coeff(E, y^N )=0 = \coeff(E, x_1^N x_2^N )$.
 By Lemma \ref{lem.y-deg}, $ E \in \BC[x_1, x_2]$. Then, by Lemma \ref{lem.x1x2}, we have $E \in \BC$, i.e. $E$ is a constant.

We will show that $E=0$. This is done by using the  inclusion of $\cH$ into $\BR^3$, which  gives a $\BC$-linear map $\iota: \CS_\xi(\BD) \to \CS_\xi(\BR^3) = \BC$. Under $\iota$, we have
\be
 E = \iota (T_N(\gamma))  - \iota (\xi^{N^2} T_N(x_1) T_N(x_2))  - \iota (\xi^{-N^2} T_N(y)). \label{eq.98}
 \ee
The right hand side involves the trivial knot and the trivial knot with framing 1, and can be calculated explicitly as follows.
 Note that $\iota(\gamma)$ is the unknot with framing 1, while $\iota(x_1)=\iota(x_2)=\iota(y)=U$, the trivial knot.
 With $T_N = S_N- S_{N-2}$, and  the framing change
 given by \eqref{eq.framing}, we find
 \be T_N(\psdraw{curl}{.25in})= (-1)^N \xi^{N^2 + 2N}\,  T_N(\psdraw{straight}{.25in})  = -\xi^{-N^2}\,  T_N(\psdraw{straight}{.25in}),
 \label{eq.100}
 \ee
 where the second identity follows from \eqref{eq.xi}. Similarly, using \eqref{eq.U}, we have
 \be
  T_N(L \sqcup U) = 2 (-1)^N \xi^{2N} \, T_N(L) = -(\xi^{2N^2} + \xi^{-2N^2}) \, T_N(L).
  \label{eq.101}
  \ee
  From \eqref{eq.100} and \eqref{eq.101}, we calculate the right hand side of \eqref{eq.98}, and find that $E=0$. This
  proves \eqref{eq:main}.

  The proof of \eqref{eq:main2} is similar. Alternatively, one can get \eqref{eq:main2} from \eqref{eq:main} by noticing that the mirror image map on $R[x_1, x_2, y]$
  is the $\BC$-algebra map sending $t$ to $t^{-1}$, leaving each of $x_1, x_2, y$ fixed.

This completes the proof of Proposition \ref{prop.main1}.

\section{Proof of Theorem \ref{thm.main2}} \label{sec.7}
Recall that $\ve= \xi^{N^2}$, where $\xi^4$ is a root of 1 of order $N$. Then $\ve^4=1$.
The map  $\CS_\ve(M) \to \CS_\xi(M)$, defined for framed links by $L\to T_N(L)$, is well-defined if and only
if it preserves the skein relations \eqref{eq.skein1} and \eqref{eq.skein2}, i.e. in $ \CS_\xi(M)$,
\begin{align}
T_N(L) = \ve T_N(L_+) + \ve^{-1} T_N(L_-)  \label{eq.001} \\
T_N(L \sqcup U) = -(\ve^2 + \ve^{-2}) T_N(L)  \label{eq.002}.
\end{align}
Here, in \eqref{eq.001}, $L, L_+, L_-$ are links appearing in the original skein relation \eqref{eq.skein1}, they are identical everywhere, except in a ball $B$, where
they look like in Figure \ref{fig.skein2}.
\begin{figure}[htpb]
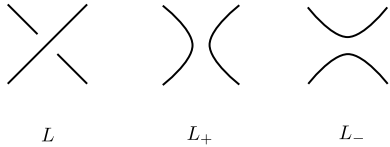

$$ \psdraw{skein}{2in} $$
\caption{The links $L$, $L_+$, and $L_-$}
\label{fig.skein2}
\end{figure}

Identity \eqref{eq.002} follows from \eqref{eq.101}. Let us prove \eqref{eq.001}.

Case 1:  The two strands of $L$ in the ball $B$ belong to the same component.  Then \eqref{eq.001} follows from Proposition \ref{prop.main1}, applied to the handlebody which is
the union of $B$ and a tubular neighborhood of $L$.

Case 2: The two strands of $L$ in $B$ belong to different components. Then the two strands of $L_+$ belong to the same component, and we can apply \eqref{eq.001} to
the case when the left hand side is $L_+$. We have
\begin{align}
T_N(L_+)  = \ T_N( \psdraw{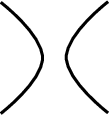}{.3in}) \ & = \ T_N( \psdraw{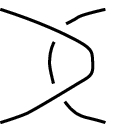}{.3in} ) \notag \\
&= \ve^{-1}\,  T_N( \psdraw{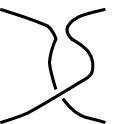}{.3in}) \ +\  \ve\, T_N(\psdraw{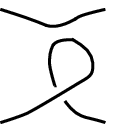}{.3in})  \label{003}\\
&= \ve^{-1}\, T_N( \psdraw{Lplus2}{.3in}) \  + \ \ve\,  (-\ve) \,T_N(  \psdraw{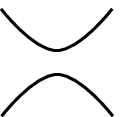}{.3in} ) \label{004}\\
&= \ve^{-1} T_N(L) -  \ve^2 T_N(L_-), \label{eq.last}
\end{align}
where \eqref{003} follows from Case 1, and \eqref{004} follows from the framing factor formula \eqref{eq.100}.
Multiplying \eqref{eq.last} by $\ve$ and using $\ve^3=\ve^{-1}$, we get \eqref{eq.001} in this case.
This completes the proof of Theorem \ref{thm.main2}.

\begin{remark}
In \cite{BW}, in order to prove Theorem \ref{thm.main2}, the authors proved, besides Proposition \ref{prop.main1}, a similar statement for links in
the cylinder over  a punctured torus. Here we bypass this extra statement by reducing the extra statement to Proposition \ref{prop.main1}. Essentially this is due to the fact
that the cylinder over  a punctured torus is the same as the cylinder over  a twice punctured disk.
\end{remark}

\end{document}